\documentclass[12pt]{amsart}
\usepackage{amsthm}
\usepackage{amssymb}
\usepackage{amsmath}
\theoremstyle{plain}
\newtheorem{proposition}{Proposition}
\newtheorem{theorem}[proposition]{Theorem}
\newtheorem{lemma}[proposition]{Lemma}
\newtheorem{corollary}[proposition]{Corollary}

\theoremstyle{definition}

\theoremstyle{definition}

\newtheorem{remark}[proposition]{Remark}

\usepackage{color}

\numberwithin{equation}{section}

\numberwithin{proposition}{section}

\gdef\myletter{}

\let\savetheequation\theequation
\def\theequation{\savetheequation\myletter}

\newcommand{\CC}{{\mathbb C}}

\newcommand{\RR}{{\mathbb R}}

\newcommand{\PP}{{\mathbb P}}

\renewcommand{\Re}{\mbox{Re}}

\renewcommand{\date}{\today}
\def \bar{\overline}

\begin{document}

\vskip 3mm

\title[Random polynomials and pluripotential ...]{\bf Random polynomials and pluripotential-theoretic extremal functions }

%\date
\author{T. Bloom* and N. Levenberg}{\thanks{*Supported in part by an NSERC of Canada grant}}
%\subjclass{32U20,\ 32U15, \ 60B99}%
%\keywords{large deviation principle, pluripotential energy}%

\address{University of Toronto, Toronto, Ontario M5S 2E4 Canada}  
\email{bloom@math.toronto.edu}

\address{Indiana University, Bloomington, IN 47405 USA}

\email{nlevenbe@indiana.edu}

\begin{abstract}
There is a natural pluripotential-theoretic extremal function $V_{K,Q}$ associated to a closed subset $K$ of $\CC^m$ and a real-valued, continuous function $Q$ on $K$. We define random polynomials $H_n$ whose coefficients with respect to a related orthonormal basis are independent, identically distributed complex-valued random variables having a very general distribution (which includes both normalized complex and real Gaussian distributions) and we prove results on a.s. convergence of a sequence $\frac{1}{n}\log |H_n|$ pointwise and in $L^1_{loc}(\CC^m)$ to $V_{K,Q}$. In addition we obtain results on a.s. convergence of a sequence of normalized zero currents $dd^c\bigl(\frac{1}{n}\log |H_n|\bigr)$ to $dd^cV_{K,Q}$ as well as asymptotics of expectations of these currents. All these results extend to random polynomial mappings and to a more general setting of positive holomorphic line bundles over a compact K\"ahler manifold.

\end{abstract}

\maketitle

\section{\bf Introduction.} \label{sec:prelim} In many probabilistic settings, one introduces {\it randomness} by considering independent, identically distributed (i.i.d.) complex-valued random variables having complex Gaussian distribution functions (appropriately normalized). We consider more general complex-valued random variables having distribution $\phi(z)dm_2(z)$ where $dm_2$ denotes Lebesgue measure on $\RR^2=\CC$. Thus $\int_{\CC}\phi(z)dm_2(z)=1$. In our setting, these random variables will be coefficients with respect to an appropriate basis of {\it random polynomials in $\CC^m$}. For such $\phi$ that are uniformly bounded on $\CC$ and whose tail probabilities $\int_{|z|\geq R}\phi(z)dm_2(z)$ decay sufficiently rapidly as $R\to \infty$, we show that almost surely (a.s.) one recovers a pluripotential-theoretic extremal function from a sequence of random polynomials (Theorem \ref{point}) and a.s. the zeros of a sequence of random polynomials converge as currents to a current formed from the potential-theoretic extremal function (Theorem \ref{l1}). We also prove, under only the tail probability hypothesis on $\phi$, a result on asymptotics of expectations  of normalized zero currents associated to random polynomials (Theorem \ref{keythm}).

We provide versions of Theorems \ref{point} and \ref{l1} in the case of polynomial mappings (Theorem \ref{pointmap}) and holomorphic line bundles over K\"ahler manifolds (Theorem \ref{kahler}) as well as a version where $K$ is unbounded (section \ref{sec:unb}). Some of these results are new even in the case of Gaussian coefficients
while some are ``universality results," i.e., they extend results known in the Gaussian case to quite general probability distributions (see also \cite{KZ}). Specifically, Theorems \ref{point} and \ref{l1} were known in the one-dimensional case (\cite{Bloom} and \cite{Bloom2}) and Corollaries \ref{expzero} and \ref{expzero1} in the Gaussian case (\cite{BS} and \cite{Bloom}). Theorem 5.2 is in \cite{SZ} in the Gaussian case. A result in the Gaussian case related to Theorem \ref{point} is in \cite{Bloom}.

 Our assumptions on $\phi$ include, in particular, the case where $\phi(z)=\phi(\Re z)$ is a {\it real} Gaussian. This latter situation was analyzed numerically by Marc Van Barel \cite{MVB} for specific extremal functions in $\CC^2$. 

The contents of the paper begin with general probabilistic preliminaries in the next section.  In section \ref{sec:pluriintro} we provide background in pluripotential theory. Section \ref{sec:efarp} utilizes these preliminaries, along with a deterministic result on limiting behavior of Bergman reproducing kernels, to give extremal function asymptotics for random polynomials. Similar analysis yields extremal function asymptotics for random polynomial mappings and sections of holomorphic line bundles over a compact K\"ahler manifold. We extend these results to unbounded sets with super-logarithmic weights in section \ref{sec:unb}. As an application in this setting, we consider random Weyl polynomials in $\CC$ and we show that, appropriately scaled, their zeros converge to normalized Lebesgue measure on the unit disk (see also \cite{KZ}). Finally, in section \ref{sec:zed}, we prove our results on asymptotics of expectations of normalized zero currents.

The first author would like to thank Andrew Stewart of the University of Toronto for helpful discussions and the authors would like to thank C. Bordenave for bringing \cite{KZ} to our attention.

\section{\bf Probabilistic preliminaries.} \label{sec:intro} We begin with a complex-valued random variable having distribution $\phi(z)dm_2(z)$ where $dm_2$ denotes Lebesgue measure on $\RR^2=\CC$. Thus $\int_{\CC}\phi(z)dm_2(z)=1$. We consider the following assumptions on $\phi$: for some $T>0$, we have
\begin{equation}\label{hyp1} |\phi(z)|\leq T \ \hbox{for all} \ z\in \CC;\end{equation}
\begin{equation}\label{hyp2} |\int_{|z|\geq R}\phi(z)dm_2(z)|\leq T/R^2 \ \hbox{for all} \ R \ \hbox{sufficiently large}.\end{equation}
If $\phi$ is real-valued, we replace $\phi(z)dm_2(z)$ by $\phi(x)dm_1(x)$ where $dm_1$ denotes Lebesgue measure on $\RR$ in (\ref{hyp1}) and (\ref{hyp2}) (and below). These hypotheses (see (1.10) in \cite{Bloom}) are very weak; for a (real or) complex Gaussian random variable with mean zero and standard deviation one, one has a tail-end estimate in (\ref{hyp2}) like $0(e^{-R^2})$. 

Next, let $Prob_n$ denote the probability measure on $\CC^n$ given by the $n-$fold product of these distributions on $\CC$; i.e., for $G\subset \CC^n$,
$$Prob_n(G):=\int_G \phi(z_1)\cdots \phi(z_n)dm_2(z_1)\cdots dm_2(z_n).$$
Finally, let $\mathcal C:= \otimes_{n=1}^{\infty} (\CC^n,Prob_n)$ be the product probability space.

We will utilize repeatedly the classical Borel-Cantelli lemma.
\begin{lemma}\label{bc} Let $\{E_n\}\subset \mathcal F$ be a sequence of events on some probability space $(\Omega, \mathcal F, \Pr)$. If the sum of the probabilities of the $E_n$ is finite, i.e.,

       $$ \sum_{n=1}^\infty \Pr(E_n)<\infty,$$
then the probability that infinitely many of them occur is $0$:

      $$ \Pr\left(\limsup_{n\to\infty} E_n\right)= \Pr\left(\bigcap_{n=1}^{\infty} \bigcup_{k=n}^{\infty} E_k\right)= 0.$$
\end{lemma}

%\noindent Here, $\limsup_{n\to\infty} E_n$ is the set of outcomes that occur infinitely many times within the infinite sequence of events $\{E_n\}$. Explicitly,

    %$$\limsup_{n\to\infty} E_n = \bigcap_{n=1}^{\infty} \bigcup_{k=n}^{\infty} E_k.$$ 
\noindent Note if $E_n'$ denotes the complement of the event $E_n$, then the conclusion is also the probability that all but finitely many of the events $\{E_n\}$ do {\it not} occur is $1$:
\begin{equation}\label{bcuse}\Pr\left(\liminf_{n\to\infty} E_n'\right) =\Pr \left(\bigcup_{n=1}^{\infty} \bigcap_{k=n}^{\infty} E_k'\right)= 1.
\end{equation}

Let $\{ w^{(n)}:=(w_{n1},...,w_{nn})\}_{n=1,2,...}$ be a sequence of non-zero vectors $w^{(n)}\in \CC^n$. We write $<\cdot,\cdot>$ for the Hermitian inner product on $\CC^n$ and $|| \cdot ||$ for the Euclidean norm. The appropriate dimension $n$ should be understood from the context. 

\begin{lemma} \label{useful2} For $\phi$ satisfying (\ref{hyp1}), let $$\mathcal A:=\{ \{a^{(n)}:=(a_{n1},...,a_{nn})\}_{n=1,2,...}\in \mathcal C: \frac{|<a^{(n)},w^{(n)}>|}{||w^{(n)}||}\geq 1/n^2$$
$$\hbox{for}  \ n \ \hbox{sufficiently large}\}.$$
Then $\mathcal A$ is of probability one in $\mathcal C$.

\end{lemma}

\begin{proof} By rescaling we may assume $||w^{(n)}||=1$. We consider 
$$Prob_n\{ a^{(n)}\in \CC^n: |<a^{(n)},w^{(n)}>|\leq 1/n^2\}$$
\begin{equation}\label{integral}=\int_{|<a^{(n)},w^{(n)}>|\leq 1/n^2} \phi(a_{n1})\cdots \phi(a_{nn})dm_2(a_{n1})\cdots dm_2(a_{nn}).\end{equation}
We may assume $|w_{n1}|\geq 1/\sqrt n$ and we make the complex-linear change of coordinates on $\CC^n$ given by:
$$\alpha_1:=a_{n1}w_{n1}+\cdots + a_{nn}w_{nn}, \ \alpha_2=a_{n2}, \cdots , \alpha_n=a_{nn}.$$
Then (\ref{integral}) becomes
$$\int_{\CC^{n-1}}\int_{|\alpha_1|\leq 1/n^2} \frac{1}{|w_{n1}|^2}\phi(\frac{\alpha_1-\alpha_2w_{n2}-\cdots - \alpha_n w_{nn}}{w_{n1}})\phi(\alpha_2)\cdots \phi(\alpha_n)$$
$$dm_2(\alpha_1)\cdots dm_2(\alpha_n).$$
Using (\ref{hyp1}) this is bounded above by
$$n|\int_{|\alpha_1|\leq 1/n^2} Tdm_2(\alpha_1)|\leq \pi T/n^3.$$
The result follows from Lemma \ref{bc}, see (\ref{bcuse}).

\end{proof}

\begin{remark} We note that the set $\mathcal A$ depends on the sequence $\{w^{(n)}\}$; but for {\it each} $\{w^{(n)}\}$, the corresponding set $\mathcal A=\mathcal A(\{w^{(n)}\})$ is of probability one in $\mathcal C$. 

\end{remark}

\begin{lemma} \label{useful3} For $\phi$ satisfying (\ref{hyp2}), let $$\mathcal A':=\{ \{a^{(n)}:=(a_{n1},...,a_{nn})\}_{n=1,2,...}\in \mathcal C: \frac{|<a^{(n)},w^{(n)}>|}{||w^{(n)}||}\leq n^2$$
$$\hbox{for}  \ n \ \hbox{sufficiently large}\}.$$
Then $\mathcal A'$ is of probability one in $\mathcal C$.

\end{lemma}

\begin{proof} By rescaling we may again assume $||w^{(n)}||=1$. Then
$$|<a^{(n)},w^{(n)}>|\leq ||a^{(n)}||\cdot ||w^{(n)}|| =||a^{(n)}||.$$
We have
$$Prob_n\{a^{(n)}\in \CC^n: ||a^{(n)}||\geq n^2\}=Prob_n\{a^{(n)}\in \CC^n: \sum_{j=1}^n|a_{nj}|^2\geq n^4\}$$
$$\leq Prob_n\{a^{(n)}\in \CC^n: |a_{nj}|\geq n^{3/2} \ \hbox{for some} \ j=1,...,n\}$$
$$=n Prob_n\{a^{(n)}\in \CC^n: |a_{n1}|\geq n^{3/2}\} \leq n\frac{T}{n^3}=\frac{T}{n^2}$$
by (\ref{hyp2}). The result again follows from Lemma \ref{bc}, see (\ref{bcuse}).

\end{proof}

\begin{remark} This time, from the proof we note that for {\it each} $\{w^{(n)}\}$, the corresponding set $\mathcal A'=\mathcal A'(\{w^{(n)}\})$ contains the same set $S$ of probability one in $\mathcal C$.  \end{remark}

Combining Lemma \ref{useful2} for (\ref{li1}) below and Lemma \ref{useful3} for (\ref{ls1}), we have:

\begin{corollary} For $\phi$ satisfying (\ref{hyp1}) and (\ref{hyp2}), with probability one in $\mathcal C$,
\begin{equation}\label{ls1}\limsup_{n\to \infty} \frac{1}{n}\log |<a^{(n)},w^{(n)}>|\leq \limsup_{n\to \infty} \frac{1}{n}\log ||w^{(n)}||\end{equation}
for all $\{w^{(n)}\}$. For each $\{w^{(n)}\}$, 
\begin{equation}\label{li1}\liminf_{n\to \infty} \frac{1}{n}\log |<a^{(n)},w^{(n)}>|\geq \liminf_{n\to \infty} \frac{1}{n}\log ||w^{(n)}||\end{equation}
with probability one in $\mathcal C$; i.e., for each $\{w^{(n)}\}$, the set 
$$\{  \{a^{(n)}:=(a_{n1},...,a_{nn})\}_{n=1,2,...}\in \mathcal C:(\ref{li1}) \ \hbox{holds}\}$$
depends on $\{w^{(n)}\}$ but is always of probability one.

\end{corollary}

We will need a version of the corollary for an appropriate subsequence $\{m(n)\}$ of the positive integers, but with the same factor $\frac{1}{n}$ in the estimates. We now let 
$$\mathcal C:= \otimes_{n=1}^{\infty} (\CC^{m(n)},Prob_{m(n)}).$$

\begin{corollary} \label{useful} For $\phi$ satisfying (\ref{hyp1}) and (\ref{hyp2}), with probability one in $\mathcal C$, if $\{m(n)\}$ is a sequence of positive integers with $m(n)=0(n^M)$ for some $M$, then 
\begin{equation}\label{limsupeqn} \limsup_{n\to \infty} \frac{1}{n}\log |<a^{(m(n))},w^{(m(n))}>|\leq \limsup_{n\to \infty} \frac{1}{n}\log ||w^{(m(n))}||\end{equation}
for all $\{w^{(m(n))}\}$. For each $\{w^{(m(n))}\}$, 
\begin{equation}\label{liminfeqn}\liminf_{n\to \infty} \frac{1}{n}\log |<a^{(m(n))},w^{(m(n))}>|\geq \liminf_{n\to \infty} \frac{1}{n}\log ||w^{(m(n))}||\end{equation}
with probability one in $\mathcal C$; i.e., for each $\{w^{(m(n))}\}$, the set 
$$\{  \{a^{(m(n))}:=(a_{m(n)1},...,a_{m(n)m(n)})\}_{n=1,2,...}\in \mathcal C:(\ref{liminfeqn}) \ \hbox{holds}\}$$
depends on $\{w^{(m(n))}\}$ but is always of probability one.

\end{corollary}

\noindent Here we simply use the fact that $\frac{1}{n}\log {n^{2M}}\to 0$ as $n\to \infty$.

For future use, we note a corollary of the proof of Lemma \ref{useful3}.

\begin{corollary} \label{forlater} For $\phi$ satisfying  (\ref{hyp2}), 
$$Prob_n\{a^{(m(n))}\in \CC^{(m(n))}: ||a^{(m(n))}||\geq n^k\}\leq T\bigl(\frac{m(n)}{n^k}\bigr)^2.$$
\end{corollary}

\section{\bf Pluripotential preliminaries.} \label{sec:pluriintro} 

A set $E\subset \CC^m$ is {\it pluripolar} if there exists a plurisubharmonic function $u\not \equiv -\infty$ with $E\subset \{z:u(z)=-\infty\}$. Pluripolar sets have $\RR^{2m}-$Lebesgue measure zero.
Our setting in this section is as follows: $K$ is a nonpluripolar compact set in $\CC^m$, $Q$ is a real-valued  continuous, function on $K$  and $\tau$ is a positive Borel measure on $K$ such that the triple $(K,Q,\tau)$ satisfies a weighted Bernstein-Markov property:
\begin{equation} \label{wtdbm} ||pe^{-nQ}||_K:=\max_{z\in K}|p(z)|e^{-nQ(z)}\leq M_n ||pe^{-nQ}||_{L^2(\tau)}\end{equation}
for all polynomials $p\in \mathcal P_n$; i.e., of degree at most $n$, where $n=1,2,...$ and 
$$\limsup_{n\to \infty} M_n^{1/n}=1.$$
If $Q\equiv 0$ we say the pair $(K,\tau)$ satisfies a Bernstein-Markov property. Note that 
$$||pe^{-nQ}||^2_{L^2(\tau)}=\int_K |p|^2e^{-2nQ}d\tau =||p||^2_{L^2(e^{-2nQ}\tau)}.$$
If $K$ is the closure of a smoothly bounded domain in $\RR^m$ or $\CC^m$ then Lebesgue measure satisfies the weighted Bernstein-Markov property for any $Q$. For a fuller discussion of the Bernstein-Markov property, see \cite{BL}, section 3.

In section \ref{sec:unb} we will consider closed but possibly {\it unbounded} sets $K$ and appropriate modifications of the hypotheses on $Q$ and the weighted Bernstein-Markov property (\ref{wtdbm}). 

We define the weighted pluricomplex Green function $V^*_{K
,Q}(z):=\limsup_{\zeta \to z}V_{K
,Q}(\zeta)$ where
\begin{equation}\label{vkq1} V_{K,Q}(z):=\sup \{\frac{1}{deg(p)}\log |p(z)|: p\in \cup_n \mathcal P_n, \ ||pe^{-nQ}||_K\leq 1\}\end{equation}
\begin{equation}\label{vkq2}=\sup \{u(z):u\in L(\CC^m), \ u\leq Q \ \hbox{on } K\}.\end{equation}
Here,  $L(\CC^m)$ is the set of all plurisubharmonic (psh) functions $u$ on $\CC^m$ with the property that $u(z) - \log |z|$ is bounded above as $|z| \to \infty$. If $Q\equiv 0$ we simply write $V_K$ and $V_K^*$. For $K$ nonpluripolar, $V^*_{K,Q}\in L(\CC^m)$ and 
$$\{z\in \CC^m: V_{K,Q}(z) < V_{K,Q}^*(z)  \}$$
is pluripolar. For example, the (unweighted) pluricomplex Green function for the $m-$torus 
$$(S^1)^m:=\{(z_1,...,z_m)\in \CC^m: |z_j|=1, \ j=1,...,m\}$$
is 
\begin{equation}\label{vtm} V_{(S^1)^m}(z)=\max_{j=1,...,m} \log^+|z_j| \end{equation}
where $\log^+|z_j|=\max[0,\log |z_j|]$.

Let $\nu$ be an $m-$multiindex and let $\{p_{\nu}^{(n)}\}_{|\nu|\leq n}$ be a set of orthonormal polynomials of degree at most $n$ in $L^2(e^{-2nQ}\tau)$ gotten by applying the Gram-Schmidt process to (a lexicographical ordering of) the monomials $\{z^{\nu}\}_{|\nu|\leq n}$. For each $n=1,2,...$ consider the corresponding Bergman kernel
$$S_n(z,\zeta):=\sum_{|\nu|\leq n} p_{\nu}^{(n)}(z)\overline{p_{\nu}^{(n)}(\zeta )}$$
and the restriction to the diagonal
\begin{equation}\label{snfcn}S_n(z,z)=\sum_{|\nu|\leq n} |p_{\nu}^{(n)}(z)|^2.\end{equation}
By the reasoning in \cite{BS}, Lemma 3.4 (or \cite{Bloom}, Lemma 2.3), we have the following.
\begin{proposition}\label{snfcnprop} Let $K\subset \CC^m$ be compact with $Q$ a real-valued, continuous function on $K$. Let $\tau$ be a positive Borel measure on $K$ such that $(K,Q,\tau)$ satisfies (\ref{wtdbm}). Then with $S_n(z,z)$ defined in (\ref{snfcn}), 
\begin{equation} \label{bmasym} \lim_{n\to \infty} \frac{1}{2n}\log S_n(z,z) = V_{K,Q}(z)\end{equation}
pointwise on $\CC^m$. If $V_{K,Q}$ is continuous, the convergence is uniform on compact subsets of $\CC^m$. 
\end{proposition}

Since we will need to modify this result in the unbounded case, we briefly indicate the two main steps in the proof. First, for each $n=1,2,...$ define
\begin{equation} \label{step1} \phi_n(z):=\sup \{ |p(z)|: p\in  \mathcal P_n, \ ||pe^{-nQ}||_K\leq 1\}.\end{equation} 
Then \begin{equation}\label{phin} \lim_{n\to \infty} \frac{1}{n}\log \phi_n(z) = V_{K,Q}(z)\end{equation} pointwise on $\CC^m$; and the convergence is uniform on compact subsets of $\CC^m$ if $V_{K,Q}$ is continuous. The next step is a comparison between $\phi_n(z)$ and $S_n(z,z)$: given $\epsilon >0$, there exists $C=C(\epsilon)>0$ independent of $n$ such that
\begin{equation} \label{step2}\frac{1}{m_n}\leq \frac{S_n(z,z)}{\phi_n(z)^2}\leq C^2(1+\epsilon)^{2n}m_n\end{equation}
where $m_n=$dim$(\mathcal P_n)= {m+n\choose n}=0(n^m)$. The left-hand inequality simply follows from the reproducing property of the Bergman kernel $S_n(z,\zeta)$ and the Cauchy-Schwarz inequality while the right-hand inequality uses (\ref{wtdbm}) (cf., Lemma 2.2 of \cite{Bloom}) .

\section{\bf Extremal function asymptotics: random polynomials.} \label{sec:efarp}

In this section $K$ is a nonpluripolar compact set in $\CC^m$; $Q$ is a real-valued, continuous function on $K$; and $\tau$ is a probability measure on $K$ such that the triple $(K,Q,\tau)$ satisfies (\ref{wtdbm}). Letting $\{p_{\nu}^{(n)}\}_{|\nu|\leq n}$ be an orthonormal basis of polynomials of degree at most $n$ in $L^2(e^{-2nQ}\tau)$ we consider random polynomials of degree at most $n$ of the form 
$$H_n(z):=\sum_{|\nu|\leq n} a_{\nu}^{(n)}p_{\nu}^{(n)}(z)$$
where the $a_{\nu}^{(n)}$ are i.i.d. complex random variables with a distribution satisfying (\ref{hyp1}) and (\ref{hyp2}). This places a probability measure $\mathcal H_n$ on $\mathcal P_n$. We form the product probability space of sequences of polynomials:
$$\mathcal H:=\otimes_{n=1}^{\infty} (\mathcal P_n,\mathcal H_n).$$
Since $m_n=$dim$(\mathcal P_n$) we can identify $\mathcal H$ with $\otimes_{n=1}^{\infty}(\CC^{m_n},Prob_{m_n})$.

\begin{theorem} \label{point} Let $a_{\nu}^{(n)}$ be i.i.d. complex random variables with a distribution satisfying (\ref{hyp1}) and (\ref{hyp2}). Then almost surely in $\mathcal H$ we have
$$\bigl(\limsup_{n\to \infty}\frac{1}{n}\log |H_n(z)|\bigr)^*=V_{K,Q}^*(z)$$
for all $z\in \CC^m$.

\end{theorem}

\begin{proof} Using the first part of Corollary \ref{useful}, (\ref{limsupeqn}), with $m(n)$ replaced by $m_n$ and
$$w^{(n)}:=p^{(n)}(z)= (p_1^{(n)}(z),...,p_{m_n}^{(n)}(z))\in \CC^{m_n},$$ almost surely in $\mathcal H$  
\begin{equation}\label{newlimsup} \limsup_{n\to \infty}\frac{1}{n}\log |H_n(z)|\leq V_{K,Q}(z)\end{equation}
for all $z\in \CC^m$ from Proposition \ref{snfcnprop}. Fix a countable dense subset $\{z_t\}_{t\in S}$ of $\CC^m$. Using the second part of Corollary \ref{useful}, (\ref{liminfeqn}), for each $z_t$, almost surely in $\mathcal H$ we have
\begin{equation}\label{newliminf}\liminf_{n\to \infty}\frac{1}{n}\log |H_n(z_t)|\geq V_{K,Q}(z_t).\end{equation}
A countable intersection of sets of probability one is a set of probability one; thus (\ref{newliminf}) holds almost surely in $\mathcal H$ for each $z_t, \ t\in S$.

Define
$$H(z):=\bigl(\limsup_{n\to \infty}\frac{1}{n}\log |H_n(z)|\bigr)^*.$$
From (\ref{newlimsup}), $H(z)\leq V_{K,Q}^*(z)$ for all $z\in \CC^m$. Moreover $H$ is plurisubharmonic; indeed, $H\in L(\CC^m)$. By (\ref{newliminf}), $H(z_t)\geq V_{K,Q}(z_t)$ for all $t\in S$. Now given $z\in \CC^m$ at which $V_{K,Q}$ is continuous, let $S'\subset S$ with $\{z_t\}_{t\in S'}$ converging to $z$. Then,
$$V_{K,Q}(z)=\lim_{t\in S', \ z_t \to z}V_{K,Q}(z_t)\leq \limsup_{t\in S', \ z_t \to z}H(z_t)\leq H(z).$$ 
Thus $H(z)=V_{K,Q}(z)$ for all $z\in \CC^m$ at which $V_{K,Q}$ is continuous. But $V_{K,Q}$ is continuous at all points except possibly for a pluripolar set, in particular a.e. in $\CC^m$. Thus $H(z)$ and $V_{K,Q}^*(z)$ are plurisubharmonic functions equal a.e. so by a general property of plurisubharmonic functions they are equal.

\end{proof}

In order to discuss convergence of linear differential operators applied to $\frac{1}{n}\log |H_n(z)|$, we need to at least have convergence to $V_{K,Q}^*(z)$ in $L_{loc}^1(\CC^m)$.

\begin{theorem} \label{l1} Let $a_{\nu}^{(n)}$ be i.i.d. complex random variables with a distribution satisfying (\ref{hyp1}) and (\ref{hyp2}). Then almost surely in $\mathcal H$ we have
$$\lim_{n\to \infty}\frac{1}{n}\log |H_n(z)|=V_{K,Q}^*(z)$$
in $L_{loc}^1(\CC^m)$ and hence
$$\lim_{n\to \infty}dd^c\bigl(\frac{1}{n}\log |H_n(z)|\bigr)=dd^cV_{K,Q}^*(z)$$
as positive currents, where $dd^c=\frac{i}{\pi}\partial \bar \partial$.
\end{theorem}

 \begin{remark} If $m=1$, $dd^cV_{K,Q}^*$ is a positive current of bidegree $(1,1)$ which can be identified with the positive measure $\frac{1}{2\pi}\Delta V_{K,Q}^*$ where $\Delta$ is the usual Laplacian, and furthermore this measure  can be identified with the 
weighted equilibrium measure for $K,Q$ \cite{ST}. For example, if $S^1=\{z\in \CC: |z|=1\}$ is the unit circle, from (\ref{vtm}), $V_{S^1}(z)=\log^+|z|:=\max[0,\log |z|]$ and
$$dd^cV_{S^1}=\frac{1}{2\pi}\Delta \log^+|z|= \frac{1}{2\pi}d\theta.$$ 
The monomials $\{z^{j}\}$ are orthonormal with respect to $\frac{1}{2\pi}d\theta$; moreover, the pair $(S^1,\frac{1}{2\pi}d\theta)$ satisfies a Bernstein-Markov property. For a random polynomial $H_n(z)=\sum_{j=0}^na_j^{(n)}z^j=a_n^{(n)}\sum_{k=1}^n(z-z^{(n)}_k)$, $$dd^c\bigl(\frac{1}{n}\log |H_n(z)|\bigr)=\frac{1}{2\pi}\Delta\bigl(\frac{1}{n}\log |H_n(z)|\bigr)=\frac{1}{n}\sum_{k=1}^n\delta_{z^{(n)}_k},$$ 
the {\it normalized zero measure} of $H_n$. In particular, if the $a_{j}^{(n)}$ are complex random variables with a distribution satisfying (\ref{hyp1}) and (\ref{hyp2}), then a.s. the normalized zero measures associated to a sequence of random polynomials $\{H_n\}$ satisfies  
$$\lim_{n\to \infty}\frac{1}{n}\sum_{k=1}^n\delta_{z^{(n)}_k}=\frac{1}{2\pi}d\theta$$
as positive measures.

\end{remark}

The proof of Theorem \ref{l1} will follow from a modification of the proof of Theorem \ref{point}  and the general deterministic result below (see also \cite{SZ}). 
\begin{theorem} \label{thmdet} Let $\{\psi_n\}\subset L(\CC^m)$. Suppose
$$\bigl(\limsup_{n\to \infty}\psi_n\bigr)^*=V(z)$$
for all $z\in \CC^m$ where $V\not \equiv 0$ and $V\in L(\CC^m)$. In addition, suppose for any subsequence $J$ of positive integers we have 
$$\bigl(\limsup_{n\in J}\psi_n\bigr)^*=V(z)$$
for all $z\in \CC^m$. Then $\psi_n \to V$ in $L_{loc}^1(\CC^m)$.

\end{theorem}
\begin{proof} The proof is by contradiction. If the conclusion is false, there is a ball $B\subset \CC^m$, an $\epsilon >0$, and a subsequence $J$ of positive integers with
\begin{equation}\label{contra} ||\psi_n -V||_{L^1(B)} \geq \epsilon, \ n\in J.\end{equation}
By Hartogs' lemma, the sequence $\{\psi_n\}_{n\in J}$ is locally bounded above (since $V$ is). Appealing to Theorem 3.2.12 of \cite{Ho}, there is a subsequence $J_1\subset J$ and $g\in L^1(B)$ with $\lim_{n\in J_1}\psi_n = g$ in $L^1(B)$. It follows from standard measure theory that there is a further subsequence $J_2 \subset J_1$ with $\lim_{n\in J_2}\psi_n(z) = g(z)$ a.e. in $B$. By assumption, 
$$\bigl(\limsup_{n\in J_2}\psi_n\bigr)^*=V(z)$$
for all $z\in \CC^m$ so that $V(z)=g(z)$ a.e. in $B$. This contradicts (\ref{contra}).

\end{proof}

\begin{proof} {\it of Theorem \ref{l1}}: From Theorem \ref{thmdet}, we need to show almost surely in $\mathcal H$ that for any subsequence $J$ of positive integers, we have
$$\bigl(\limsup_{n\in J}\frac{1}{n}\log |H_n(z)|\bigr)^*=V_{K,Q}^*(z)$$
for all $z\in \CC^m$. Fix any subsequence $J$. Following the proof of Theorem \ref{point}, almost surely in $\mathcal H$  
$$\limsup_{n\in J}\frac{1}{n}\log |H_n(z)|\leq \limsup_{n\to \infty}\frac{1}{n}\log |H_n(z)|\leq V_{K,Q}(z)$$
for all $z\in \CC^m$ from (\ref{newlimsup}) and the fact that $J$ is a subsequence of positive integers. Fix a countable dense subset $\{z_t\}_{t\in S}$ of $\CC^m$.  Then for each $z_t$, almost surely in $\mathcal H$ we have
$$\liminf_{n\in J}\frac{1}{n}\log |H_n(z)|\geq \liminf_{n\to \infty}\frac{1}{n}\log |H_n(z_t)|\geq V_{K,Q}(z_t)$$
from (\ref{newliminf}) and the fact that $J$ is a subsequence of positive integers. This relation holds almost surely in $\mathcal H$ for each $z_t, \ t\in S$. 

Now define
$$H_J(z):=\bigl(\limsup_{n\in J}\frac{1}{n}\log |H_n(z)|\bigr)^*.$$
Then $H_J(z)\leq V_{K,Q}^*(z)$ for all $z\in \CC^m$; $H_J$ is plurisubharmonic; indeed, $H_J\in L(\CC^m)$; and $H_J(z_t)\geq V_{K,Q}(z_t)$ for all $t\in S$. Given $z\in \CC^m$ at which $V_{K,Q}$ is continuous, let $S'\subset S$ with $\{z_t\}_{t\in S'}$ converging to $z$. Then
$$V_{K,Q}(z)=\lim_{t\in S', \ z_t \to z}V_{K,Q}(z_t)\leq \limsup_{t\in S', \ z_t \to z} H_J(z_t)\leq H_J(z).$$ 
Thus $H_J(z)=V_{K,Q}(z)$ for all $z\in \CC^m$ at which $V_{K,Q}$ is continuous and hence everywhere.

\end{proof}

 \begin{remark} From (\ref{vtm}), for the $m-$torus $(S^1)^m\subset \CC^m$, 
$$V_{(S^1)^m}(z)=\max_{j=1,...,m} \log^+|z_j|$$ and the monomials $\{z^{\nu}=z_1^{\nu_1}\cdots z_m^{\nu_m}\}$ are orthonormal with respect to the measure $$ \frac{1}{2\pi}d\theta_1 \cdots \frac{1}{2\pi}d\theta_m.$$ 
The pair $((S^1)^m,\frac{1}{2\pi}d\theta_1 \cdots \frac{1}{2\pi}d\theta_m)$ satisfies a Bernstein-Markov property. Thus a.s. a sequence of random polynomials $\{H_n\}$ of the form 
$$H_n(z):=\sum_{|\nu|\leq n} a_{\nu}^{(n)}z^{\nu},$$
where the $a_{\nu}^{(n)}$ are complex random variables with a distribution satisfying (\ref{hyp1}) and (\ref{hyp2}), satisfies  
$$\lim_{n\to \infty}\frac{1}{n}\log |H_n(z)|=\max_{j=1,...,m} \log^+|z_j|$$
pointwise a.e. for $z\in \CC^m$ and in $L_{loc}^1(\CC^m)$.

 \begin{remark} We emphasize that in Theorems \ref{point} and \ref{l1} the probability space $\mathcal{H}$ depends on $\tau$ but the weighted 
pluricomplex Green function $V_{K,Q}^*$ depends only on $K$ and $Q$. 
\end{remark}

\end{remark}

\section {\bf Extensions.} \label{sec:ext} 

Theorems \ref{point} and \ref{l1} remain valid for random polynomial {\it mappings}. Precisely, the set-up is as follows. For any $k=1,...,m$ and for each $n=1,2,...$ we consider $k-$tuples $F_n:=(H_n^{(1)}(z),...,H_n^{(k)}(z))$ of random polynomials of degree at most $n$, i.e., random polynomial mappings $(\mathcal P_n)^k$ of degree at most $n$, where each $H_n^{(j)}$ for $j=1,...,k$ is of the form 
\begin{equation}\label{ranmap} H_n^{(j)}(z):=\sum_{|\nu|\leq n} a_{\nu}^{(n,j)}p_{\nu}^{(n)}(z)\end{equation}
and the $a_{\nu}^{(n,j)}$ are complex random variables with a distribution satisfying (\ref{hyp1}) and (\ref{hyp2}). This places a probability measure $\mathcal F_n$ on $(\mathcal P_n)^k$. We form the product probability space of sequences of these polynomial mappings:
$$\mathcal F:=\otimes_{n=1}^{\infty} ((\mathcal P_n)^k,\mathcal F_n).$$
Here we can identify $\mathcal F$ with $\otimes_{n=1}^{\infty}((\CC^{m_n})^k,(Prob_{m_n})^k)$.

Writing $||F_n(z)||^2 := \sum_{j=1}^k|H_n^{(j)}(z)|^2$, we have the following generalizations of Theorems \ref{point} and \ref{l1}. 

\begin{theorem} \label{pointmap} For any $k=1,...,m$ let $a_{\nu}^{(n,j)}$ for $j=1,...,k$ and $n=1,2,...$ be i.i.d. complex random variables with a distribution satisfying (\ref{hyp1}) and (\ref{hyp2}). Then almost surely in $\mathcal F$ we have
$$\bigl(\limsup_{n\to \infty}\frac{1}{n}\log ||F_n(z)||\bigr)^*=V_{K,Q}^*(z)$$
pointwise for all $z\in \CC^m$ and
$$\lim_{n\to \infty}\frac{1}{n}\log ||F_n(z)||=V_{K,Q}^*(z)$$
in $L_{loc}^1(\CC^m)$. Hence
$$\lim_{n\to \infty}dd^c\bigl(\frac{1}{n}\log ||F_n(z)||\bigr)=dd^cV_{K,Q}^*(z)$$
as positive currents.
\end{theorem}

\noindent  Again, the probability space $\mathcal{F}$ depends on $\tau$ but the weighted 
pluricomplex Green function $V_{K,Q}^*$ depends only on $K$ and $Q$. 

Theorems \ref{point}, \ref{l1}, and \ref{pointmap} extend to the more general setting of positive holomorphic line bundles over compact K\"ahler manifolds. Let $X$ be a compact, complex manifold of (complex) dimension $m$ equipped with a K\"ahler form $\omega$. Fixing a volume form on $X$ (e.g., $\omega^m$), define
$$PSH(X,\omega):=\{\Phi\in L^1(X): \Phi \ \hbox{uppersemicontinuous and} \ dd^c\Phi +\omega \geq 0\},$$
the class of {\it $\omega-$plurisubharmonic functions} on $X$. In the case of $X=\PP^m$ ($m-$complex dimensional projective space) with the Fubini-Study form $\omega_{FS}$, there is a one-to-one correspondence between $PSH(\PP^m,\omega_{FS})$ and $L(\CC^m)$. Indeed, identifying $\CC^m$ with the affine subset of $\PP^m$ given by the set of points $\{[1:z_1:\cdots:z_m]\}$ in homogeneous coordinates, if $\Phi \in PSH(\PP^m,\omega_{FS})$, then 
\begin{equation} \label{pshtol} u(z)=u(z_1,...,z_m):= \Phi ([1:z_1:\cdots:z_m])+u_0(z)\in  L(\CC^m)\end{equation}
where $u_0(z):=\frac{1}{2}\log (1+|z|^2)$.

Let ${\mathcal L}$ be a holomorphic line bundle over $X$. For a smooth hermitian metric $\psi=\{\psi_i\}$ on ${\mathcal L}$, where $\psi_i$ are defined on a trivializing open cover $\{U_i\}$ of $X$, $dd^c\psi$ is a globally defined $(1,1)-$form on $X$, called the curvature form of $\psi$. The line bundle with this metric is {\it positive} if the curvature form is positive. Using the notation ${\mathcal L}^n$ for the $n-$th tensor power of ${\mathcal L}$, the space of global holomorphic sections $H^0(X,{\mathcal L}^n)$ of ${\mathcal L}^n$ is known to have dimension $0(n^m)$ if ${\mathcal L}$ is positive.

As an example, if one takes the hyperplane section bundle ${\mathcal O}(1)$ over $X= \PP^m$ and one uses the Fubini-Study metric $\psi_{FS}$ on $\mathcal O(1)$, $dd^c\psi_{FS}$ is the Fubini-Study K\"ahler form $\omega_{FS}$ on $ \PP^m$. The elements of $H^0( \PP^m,{\mathcal O}(n))$ can be identified with homogeneous polynomials on $\CC^{m+1}$ of degree $n$ (or SU($m+1$) polynomials; cf., \cite{SZ}). These can naturally be put in one-to-one correspondence with $\mathcal P_n$.

For $f,g\in H^0(X,{\mathcal L})$, we have the pointwise inner product
$$<f(x),g(x)>_{\psi(x)}:=f_i(x)\overline {g_i(x)}e^{-\psi_i(x)}$$ on $U_i\subset X$. Similarly, the pointwise inner product on $H^0(X,{\mathcal L}^n)$ is  
$$<f_n(x),g_n(x)>_{n\psi(x)}:=(f_n)_i(x)\overline{ (g_n)_i(x)}e^{-n\psi_i(x)}$$
on $U_i$ and we write $||f_n(x)||_{n\psi(x)}:= \sqrt {<f_n(x),f_n(x)>_{n\psi(x)}}$.

In this setting, given a closed subset $K$ of $X$, and given a continuous function $q$ on $K$, one defines a weighted global extremal function
$${\mathcal V}_{K,q}(x):=\sup\{\Phi(x): \Phi \in PSH(X,\omega), \ \Phi \leq q \ \hbox{on} \ K\}.$$
From (\ref{pshtol}), if $X=\PP^m$ with the Fubini-Study form $\omega_{FS}$, for a compact set $K\subset \CC^m\subset \PP^m$ it follows that 
$${\mathcal V}_{K,q}([1:z_1:\cdots:z_m])=V_{K,(u_0+q)|_K}(z)-u_0(z).$$
For the ``unweighted'' case, $q\equiv 0$, the proof of Theorem 6.2 in \cite{GZ} shows that $\mathcal V_{K,0}(x)$ coincides with
$$\sup_{n=1,2,...} \sup\{\frac{1}{n}\log ||f_n(x)||_{n\psi(x)}: f_n \in H^0(X,{\mathcal L}^n), \ \max_{x\in K}||f_n(x)||_{n\psi(x)}\leq 1\}$$ 
i.e., we have the analogue of the equality of (\ref{vkq1}) with (\ref{vkq2}) in this setting. The proof of Lemma 3.2 in \cite{BS} also carries over to verify the analogue of (\ref{phin}). In the ``weighted'' case, given a continuous $q$ on $K$, the proof of Theorem 6.2 in \cite{GZ} carries over to show ${\mathcal V}_{K,q}(x)$ coincides with
$$\sup_{n=1,2,...} \sup\{\frac{1}{n}\log ||f_n(x)||_{n\psi(x)}: f_n \in H^0(X,{\mathcal L}^n), $$
$$\ \max_{x\in K}\bigl(||f_n(x)||_{n\psi(x)}e^{-nq(x)}\bigr)\leq 1\}$$
and as is pointed out in Lemma 2.1 of \cite{Bloom} in the $\CC^m$ setting, the analogue of (\ref{phin}) (see (\ref{step1})) carries over in the weighted situation.

Given $K$ and $q$ together with a measure $\tau$ on $K$, we say that the triple $(K,q,\tau)$ satisfies a weighted Bernstein-Markov property as in (\ref{wtdbm}) if 
$$ \max_{x\in K}\bigl(||f_n(x)||_{n\psi(x)}e^{-nq(x)}\bigr)\leq M_n \bigl(\int_K[||f_n(x)||_{n\psi(x)}e^{-nq(x)}]^2_{n\psi(x)}d\tau(x)\bigr)^{1/2}$$
for all $f_n\in H^0(X,{\mathcal L}^n)$ where $\limsup_{n\to \infty} M_n^{1/n}=1$. 
Such a measure induces a nondegenerate weighted $L^2-$norm and inner product on $H^0(X,{\mathcal L}^n)$: given $f_n,g_n\in H^0(X,{\mathcal L}^n)$,
$$<f_n,g_n>_{n\tau,q}:=\int_K <f_n(x),g_n(x)>_{n\psi(x)}e^{-2nq(x)}d\tau(x)$$
and
$$||f_n||^2_{n\tau,q}:=\int_K ||f_n(x)||^2_{n\psi(x)}e^{-2nq(x)}d\tau(x).$$
Letting $N(n)$ be the dimension of $H^0(X,{\mathcal L}^n)$, we take an orthonormal basis $\{s_{j}^{(n)}\}_{j=1}^{N(n)}$ of $H^0(X,{\mathcal L}^n)$ with respect to this inner product and we consider random sections in $H^0(X,{\mathcal L}^n)$ of the form 
$$H_n(x):=\sum_{j=1}^{N(n)} a_{j}^{(n)}s_{j}^{(n)}(x)$$
where the $a_{j}^{(n)}$ are i.i.d. complex random variables with a distribution satisfying (\ref{hyp1}) and (\ref{hyp2}). This places a probability measure $\mathcal H_n$ on $H^0(X,{\mathcal L}^n)$. We form the product probability space of sequences of sections:
$$\mathcal H:=\otimes_{n=1}^{\infty} (H^0(X,{\mathcal L}^n),\mathcal H_n).$$

We can also form the Bergman kernels 
$$S_n(x,y):= \sum_{j=1}^{N(n)} s_{j}^{(n)}(x)\otimes \bar s_{j}^{(n)}(y)$$ for $H^0(X,{\mathcal L}^n)$;  this defines $S_n$ as a section of the line bundle ${\mathcal L}^n \otimes \overline {\mathcal L}^n$ over $X\times X$. The analogue of (\ref{step2}) holds with $m_n$ replaced by $N(n)$ (cf., Lemmas 2.2 and 2.3 in \cite{Bloom}). Since $N(n)=0(n^m)$, we have the following result.

\begin{theorem} \label{kahler} Let $a_{j}^{(n)}$ for $j=1,...,N(n)$ and $n=1,2,...$ be i.i.d. complex random variables with a distribution satisfying (\ref{hyp1}) and (\ref{hyp2}). Then almost surely in $\mathcal H$ we have
$$\bigl(\limsup_{n\to \infty}\frac{1}{n}\log |H_n(x)|\bigr)^*={\mathcal V}_{K,q}^*(x)$$
pointwise for all $x\in X$ and
$$\lim_{n\to \infty}\frac{1}{n}\log |H_n(x)|={\mathcal V}_{K,q}^*(x)$$
in $L^1(X)$. Hence
$$\lim_{n\to \infty}\bigl[dd^c\bigl(\frac{1}{n}\log |H_n|\bigr)+\omega\bigr]=dd^c{\mathcal V}_{K,q}^*+\omega$$
as positive currents.
\end{theorem}

Since $X$ is compact and the currents $dd^c\bigl(\frac{1}{n}\log |H_n|\bigr)+\omega, \ dd^c{\mathcal V}_{K,q}^*+\omega$ are positive, from a well-known result we obtain a strengthening of the last statement in Theorem \ref{kahler},  generalizing Theorem 1.1 of \cite{SZ}. We write $\bigl<\Psi,\alpha\bigr>_X$ for the action of a $(1,1)$ current $\Psi$ with measure coefficients on an $(m-1,m-1)$ form $\alpha$ with continuous coefficients.

\begin{corollary} Under the hypotheses of Theorem \ref{kahler}, almost surely in $\mathcal H$ we have
$$\lim_{n\to \infty}\bigl[dd^c\bigl(\frac{1}{n}\log |H_n|\bigr)+\omega\bigr]=dd^c{\mathcal V}_{K,q}^*+\omega$$
in the sense of measures; i.e., for all $(m-1,m-1)$ forms $\alpha$ with continuous coefficients, 
$$\lim_{n\to \infty}\bigl<\bigl[dd^c\bigl(\frac{1}{n}\log |H_n|\bigr)+\omega\bigr],\alpha\bigr>_X=\bigl<dd^c{\mathcal V}_{K,q}^*+\omega,\alpha\bigr>_X.$$

\end{corollary}

\section{\bf The unbounded case.} \label{sec:unb} 
We will show how to extend Theorems \ref{point}, \ref{l1}, and \ref{pointmap} to the case where we replace the compact set $K$ by an unbounded set $Y$. Thus let $Y$ be a closed, unbounded subset of $\CC^m$. We let $Q$ be a continuous, real-valued, {\it super-logarithmic} function on $Y$: for some $b>0$,
\begin{equation}\label{u1}\lim_{|z|\to\infty}\bigl(Q(z)-(1+b)\log|z|\bigr)=+\infty. \end{equation}
For $r>0$ we let $Y_r:=\{z\in Y: |z|\leq r\}$. Then for $r$ sufficiently large $V_{Y,Q}=V_{Y_r,Q}$ (cf., \cite{ST}, Appendix B, Lemma 2.2). Thus if we let $S_Q$ denote the support of the  {\it Monge-Amp\`ere measure} $(dd^cV^*_{Y,Q})^m$, it follows that $S_Q$ is compact.

Let $\tau$ be a locally finite positive Borel measure on $Y$ satisfying
\begin{equation}\label{u2}\int_Y\frac{1}{|z|^a}d\tau(z) <+\infty\end{equation}
for some $a>0$ and also such that $(Y_r,Q,\tau)$ satisfies the weighted Bernstein-Markov inequality (\ref{wtdbm}) for $r$ sufficiently large.  An important example is Lebesgue measure on $\RR^m$ or $\CC^m$ with $Q(z)=|z|^2/2$.

Because of (\ref{u1}) and (\ref{u2}) all moments of the measure  $e^{-nQ}d\tau$ of order at most $n$ are finite. Thus we may apply the Gram-Schmidt
orthogonalization procedure to a lexicographic ordering of the monomials of degree at most $n$ to obtain orthonormal polynomials $p^{(n)}_{\nu}$. That is 
\begin{equation}\label{u4}\int_Yp^{(n)}_{\nu}(z)\overline{p^{(n)}_{\gamma}(z)}e^{-2nQ}d\tau =\delta_{\nu ,\gamma}\end{equation}
where $\nu$ and $\gamma$ are multi-indices. As before, we consider random polynomials of degree at most $n$ of the form 
$$H_n(z):=\sum_{|\nu|\leq n} a_{\nu}^{(n)}p_{\nu}^{(n)}(z)$$
where the $a_{\nu}^{(n)}$ are i.i.d. complex random variables with a distribution satisfying (\ref{hyp1}) and (\ref{hyp2}).
Our goal is to prove Theorems \ref{point}, \ref{l1}, and \ref{pointmap} in this context.

Aiming for a version of Proposition \ref{snfcnprop}, for each $n=1,2,...$ we define an analogue of (\ref{step1}):
\begin{equation}\label{phinq}\phi^{(n)}_{Y,Q}(z):=\sup\{|g(z)|:g\in\mathcal{P}_n,\ \text{and} \ ||ge^{-nQ}||_Y\leq 1\}.\end{equation}
 It is known that {\it weighted polynomials} $ge^{-nQ}$ assume their maximum modulus on $Y$ on the set $S_Q$. In particular, for $r$ such that $S_Q\subset Y_r$ we have
\begin{equation}\label{u4}||ge^{-nQ}||_Y=||ge^{-nQ}||_{Y_r}, \ g\in\mathcal{P}_n, \ \hbox{for each} \ n.\end{equation}
Moreover, it is a result of Siciak that 
$$\lim_{n\to \infty}\frac{1}{n}\log\phi^{(n)}_{Y,Q}(z)=V_{Y,Q}(z)$$ 
pointwise on $\CC^m$. Thus for any $r$ such that $S_Q\subset Y_r$, we have $V_{Y,Q}=V_{Y_r,Q}$.

Next, for each $n=1,2,...$ consider the corresponding Bergman kernel
$$S_n(z,\zeta):=\sum_{|\nu|\leq n} p_{\nu}^{(n)}(z)\overline{p_{\nu}^{(n)}(\zeta )}$$
and the restriction to the diagonal
$$S_n(z,z)=\sum_{|\nu|\leq n} |p_{\nu}^{(n)}(z)|^2.$$
We must prove a version of (\ref{bmasym}): 
\begin{equation}\label{est} \lim_{n\to \infty} \frac{1}{2n}\log S_n(z,z) = V_{Y,Q}(z)\end{equation}
pointwise on $\CC^m$.
We need the following estimate, which utilizes (\ref{u1}), (\ref{u2}), and the weighted Bernstein-Markov inequality of $(Y_r,Q,\tau)$ for $r$ sufficiently large. 

\begin{lemma} \label{lemunb} Given $\beta>0$ and $U$ a relatively open subset of $Y$ such that $S_Q\subset U$, there is a constant $c>0$ independent of $n$ and $g$ such that
$$\int_Y|e^{-nQ}g|^{\beta}d\tau\leq (1+O(e^{-nc}))\int_U|e^{-nQ}g|^{\beta}d\tau$$
for all $g\in\mathcal{P}_n$.\end{lemma}
\begin{proof}This is theorem 6.1 in \cite{Bp} in the case $m=1$; the proof in $\CC^m$ for $m>1$ is identical. \end{proof}
The relation (\ref{est}) in this setting will be a consequence of the following estimate,  analogous to (\ref{step2}).
\begin{theorem} \label{thmunb} Given $\epsilon >0$ there exist constants $C_1, C_2>0$ independent of $n$ such that
$$\frac{1}{C_1(1+\epsilon)^{2n}m_n}\leq\frac{S_n(z,z)}{(\phi^{(n)}_{Y,Q}(z))^2}\leq C_2(1+\epsilon)^{2n}m_n$$\end{theorem} 
\begin{proof}
For any $g\in\mathcal{P}_n$, from (\ref{phinq}),  
\begin{equation}\label{Geqn}\frac{|g(z)|}{||g(z)e^{-nQ(z)}||_Y}\leq \phi^{(n)}_{Y,Q}(z).\end{equation}
We will apply (\ref{Geqn}) to the orthonormal polynomials $p^{(n)}_{\nu}(z)$. We have 
$$\int_Y|p^{(n)}_{\nu}(z)e^{-nQ(z)}|^2d\tau=1$$
so for $r>0$
$$\int_{Y_r}|p^{(n)}_{\nu}(z)e^{-nQ(z)}|^2d\tau \leq 1.$$
Fixing $r$ sufficiently large so that $(Y_r,Q,\tau)$ satisfies a weighted Bernstein-Markov inequality and $S_Q\subset Y_r$, using (\ref{u4}) we have, given $\epsilon>0$, 
\begin{equation}\label{u6}||p^{(n)}_{\nu}e^{-nQ}||_{Y}=||p^{(n)}_{\nu}e^{-nQ}||_{Y_r}\leq C(1+\epsilon)^n.\end{equation}
Then (\ref{Geqn}) with $p^{(n)}_{\nu}$ yields
$$|p^{(n)}_{\nu}(z)|^2\leq C^2(1+\epsilon )^{2n}(\phi^{(n)}_{Y,Q}(z))^2,$$
giving the right-hand inequality in Theorem \ref{thmunb}. 
 
To prove the left-hand inequality we proceed as follows. First note that any $g\in\mathcal{P}_n$ can be written uniquely in the form
$$g(z)=\sum_{|\nu |\leq n}t_{\nu}p^{(n)}_{\nu}(z)$$ where $t_{\nu}\in\CC$ and 
$$t_{\nu}=\int_Yg(z)\overline{p^{(n)}_{\nu}(z)}e^{-2nQ(z)}d\tau.$$
If $||ge^{-nQ}||_Y\leq 1$ then we have
$$|t_{\nu}|\leq\int_Y|p^{(n)}_{\nu}(z)|e^{-nQ(z)}d\tau.$$
Fix $r$ sufficiently large so that $(Y_r,Q,\tau)$ satisfies a weighted Bernstein-Markov inequality and $S_Q\subset Y_r$. By Lemma \ref{lemunb} we get 
$$|t_{\nu}|\leq (1+O(e^{-nc}))\int_{Y_r}|p^{(n)}_{\nu}(z)|e^{-nQ(z)}d\tau.$$
 Using (\ref{u6}) we get for any $\epsilon >0$
$$|t_{\nu}|\leq C(1+\epsilon )^n\tau(Y_r)=\tilde C(1+\epsilon )^n$$
since $\tau$ is locally finite. 
Hence $$|g(z)|\leq \sum_{|\nu|\leq n}|t_{\nu}||p^{(n)}_{\nu}(z)|\leq  \tilde C(1+\epsilon )^n\cdot m_n^{1/2}S_n(z,z)^{1/2}$$and the left-hand inequality follows from the definition of $\phi^{(n)}_{Y,Q}(z)$ in (\ref{phinq}).
\end{proof}

Theorems \ref{point}, \ref{l1}, and \ref{pointmap} in this setting follow immediately. As an application, we consider random Weyl polynomials in $\CC$ (cf., \cite{BC}). These are random polynomials in one variable of the form
$$\sum_{j=0}^na_{nj}\frac{z^j}{\sqrt{j!}}$$
where the $a_{nj}$ are i.i.d. complex random variables and we assume that their distributions  satisfy (\ref{hyp1}) and (\ref{hyp2}). We will show that appropriately scaled, the zeros converge to normalized Lebesgue measure on the unit disk in the plane. This result has also been obtained in \cite{KZ} as a special case of a more general result.

Scaling the zeros by $1/\sqrt{n}$ we consider the polynomials 
$$H_n(z)=\sum_{j=0}^na_{nj}\frac{z^j\sqrt{n^j}}{\sqrt{j!}}.$$
Now consider the weight function $Q(z)=\frac{|z|^2}{2}$ on the set $Y=\CC$ and let $\tau$ be Lebesgue measure on $\CC$. The function $Q$ is super-logarithmic and radial. The weighted equilibrium measure is calculated on page 245 in \cite{ST} and is  
normalized Lebesgue measure on the unit disk. Since the weight function $Q$ and the measure $\tau$ are radial, the orthonormal polynomials $\{p_{nj}\}_{j=0,...,n}$ are monomials and a routine calculation shows that they are, in fact 
$$p_{nj}(z)= \frac{\sqrt \pi}{\sqrt {2n}}\frac{z^j\sqrt{n^j}}{\sqrt{j!}}.$$
The result now follows from Theorem \ref{l1} in the unbounded case.

\section{\bf Expected zero distribution.} \label{sec:zed}  In this section, we return to the setting of section \ref{sec:efarp}, with $K\subset \CC^m$ compact; $Q$ a real-valued, continuous function on $K$; and $\tau$ a probability measure on $K$ such that the triple $(K,Q,\tau)$ satisfies (\ref{wtdbm}). Here we will assume $V_{K,Q}$ is continuous. If $Q\equiv 0$, the set $K$ is called {\it $L-$regular} if $V_K$ is continuous; for $Q\not \equiv 0$, a sufficient condition for continuity of $V_{K,Q}$ is {\it local $L-$regularity} of $K$. We refer the reader to \cite{K} for more on these notions. 

If $\{p_{\nu}^{(n)}\}_{|\nu|\leq n}$ form an orthonormal basis of polynomials of degree at most $n$ in $L^2(e^{-2nQ}\tau)$, we consider random polynomials of degree at most $n$ of the form 
$$H_n(z):=\sum_{|\nu|\leq n} a_{\nu}^{(n)}p_{\nu}^{(n)}(z)$$
where the $a_{\nu}^{(n)}$ are i.i.d. complex random variables with distribution $\phi$. We emphasize that {\it the results of this section require that $\phi$ satisfy (\ref{hyp2}) only.}

 Let 
$$Z_{H_n}=dd^c\log |H_n(z)|$$
be the zero current of $H_n$ and
$$\tilde Z_{H_n}=\frac{1}{n}dd^c\log |H_n(z)|$$
be the normalized zero current of $H_n$. 
For the common zeros of $k$ polynomials $H_n^{(1)},...,H_n^{(k)}$ where, as in (\ref{ranmap}), $$H_n^{(j)}(z):=\sum_{|\nu|\leq n} a_{\nu}^{(n,j)}p_{\nu}^{(n)}(z),$$ we write
$$Z^k_{{\bf H}_n}=dd^c\log |H_n^{(1)}(z)|\wedge \cdots \wedge dd^c\log |H_n^{(k)}(z)|$$
for the zero current and
$$\tilde Z^k_{{\bf H}_n}=\frac{1}{n^k}dd^c\log |H_n^{(1)}(z)|\wedge \cdots \wedge dd^c\log |H_n^{(k)}(z)|$$
for the normalized zero current. These are a.s. well-defined. We are interested in the expectation $E(\tilde Z^k_{{\bf H}_n})$ of the normalized zero current $\tilde Z^k_{{\bf H}_n}$. This is itself a positive current of bidegree $(k,k)$. For $k=1$, the action of the $(1,1)$ current $E(\tilde Z_{H_n})$ on an $(m-1,m-1)$ form $\alpha$ with $C_0^{\infty}(\CC^m)$ coefficients is given as the average of the action $\bigl(\tilde Z_{H_n},\alpha \bigr)$ of the normalized zero current on $\alpha$:
$$\bigl(E(\tilde Z_{H_n}),\alpha\bigr):=\int_{\CC^{m_n}}\bigl(\tilde Z_{H_n},\alpha \bigr)dProb_{m_n}(a^{(n)})$$
$$=\int_{\CC^{m_n}}\bigl(\frac{1}{n}dd^c\log |H_n(z)|,\alpha \bigr)dProb_{m_n}(a^{(n)}).$$
The deterministic result for the weighted Bergman kernels 
$$S_n(z,z)=\sum_{|\nu|\leq n} |p_{\nu}^{(n)}(z)|^2$$
(see \ref{snfcn}) that we will use is from (\ref{bmasym}): 
$$\lim_{n\to \infty} \frac{1}{2n}\log S_n(z,z) = V_{K,Q}(z)$$ uniformly on compact subsets of $\CC^m$ since $V_{K,Q}$ is continuous. This implies 
\begin{equation}\label{maconv}\lim_{n\to \infty} \bigl(dd^c\frac{1}{2n}\log S_n(z,z)\bigr)^k = \bigl(dd^cV_{K,Q}(z)\bigr)^k\end{equation}
as positive currents for $k=1,2,...,m$. We remark that for $k=1$ all one needs is $L^1_{loc}(\CC^m)$ convergence of $\frac{1}{2n}\log S_n$ to $V_{K,Q}$ for (\ref{maconv}).

We write $p^{(n)}(z)$ for the $m_n-$tuple $\{p_{\nu}^{(n)}(z)\}_{|\nu|\leq n}$ and $a^{(n)}$ for the $m_n-$tuple $\{a_{\nu}^{(n)}\}$. Following \cite{SZ}, we write ${\bf u}(z):=\frac{p^{(n)}(z)}{\sqrt {S_n(z,z)}}$. Then
\begin{equation}\label{eeqn}\bigl(E(\tilde Z_{H_n}),\alpha\bigr)=\int_{\CC^{m_n}}\bigl(\frac{1}{n}dd^c\log |<a^{(n)},{\bf u}(z)>|,\alpha\bigr)dProb_{m_n}(a^{(n)})\end{equation}
$$+\int_{\CC^{m_n}}\bigl(\frac{1}{2n}dd^c\log S_n(z,z),\alpha\bigr)dProb_{m_n}(a^{(n)}).$$
By (\ref{maconv}) with $k=1$, this last term tends to $\bigl(dd^cV_{K,Q}(z),\alpha \bigr)$ as $n\to \infty$. From the definition of $dd^c\log |<a^{(n)},{\bf u}(z)>|$ as a positive current of bidegree $(1,1)$,
%Using integration by parts on the first term on the right side in \ref{eeqn},we may replace
$$\bigl(dd^c\log |<a^{(n)},{\bf u}(z)>|,\alpha\bigr)=\bigl(\log |<a^{(n)},{\bf u}(z)>|,dd^c\alpha\bigr).$$
Using this in the first term in (\ref{eeqn}) we immediately have the following result.

\begin{proposition} \label{propr} Let $Prob_{m_n}$ be a probability distribution and for $n=1,2,...$ define $$I_n({\bf u}):=\int_{\CC^{m_n}}\bigl(\log |<a^{(n)},{\bf u}>|\bigr)dProb_{m_n}(a^{(n)})$$
for ${\bf u}\in \CC^{m_n}$. Suppose there exists a constant $C$ independent of $n$ with
\begin{equation}\label{reqn}| I_n({\bf u})| \leq C \log n \ \hbox{for all} \ {\bf u} \ \hbox{with} \ |{\bf u}|=1.\end{equation}
Then 
\begin{equation}\label{expzeroE}\lim_{n\to \infty} E(\tilde Z_{H_n})=dd^cV_{K,Q}.\end{equation}

\end{proposition}

We verify (\ref{reqn}) for distributions $\phi$ satisfying  (\ref{hyp2}). 

\begin{theorem} \label{keythm} Let $a_{\nu}^{(n)}$ be i.i.d. complex random variables with a distribution satisfying  (\ref{hyp2}). Then (\ref{reqn}) holds and hence $\lim_{n\to \infty} E(\tilde Z_{H_n})=dd^cV_{K,Q}$.
\end{theorem}

\begin{remark}
In the {\it complex} Gaussian case, i.e., where $a_{\nu}^{(n)}$ are i.i.d. complex Gaussians with mean zero and variance one so that
$$dProb_{m_n}(a^{(n)})=\frac{1}{\sqrt{2\pi}^{m_n}}e^{-|a^{(n)}|^2}\prod_{|\nu|\leq n}dm_2(a_{\nu}^{(n)}),$$
we have 
$$I_n({\bf u})=\frac{1}{\sqrt {2\pi}}\int_{\CC}(\log |z|)e^{-|z|^2}dm_2(z)$$
if  $|{\bf u}|=1$ by the unitary invariance of the integral. Thus, in this case, $I_n({\bf u})$ is {\it equal} to a constant independent of $n$ for $|{\bf u}|=1$. 
\end{remark}

\begin{proof} Fix $n$ and write ${\bf a}:=a^{(n)}$. For ${\bf u}$ fixed with $|{\bf u}|=1$ we set $g({\bf a}):=\log |<{\bf a},{\bf u}>|$. Thus we want to show 
\begin{equation}\label{geqn}\Big|\int_{\CC^{m_n}}g({\bf a})dProb_{m_n}({\bf a})\Big|=0(\log n)\end{equation}
independent of ${\bf u}$.

To prove (\ref{geqn}), it suffices to show 
\begin{equation} \label{geqn2} \int_{\{g > (m+1)\log n\}}g({\bf a})dProb_{m_n}({\bf a})=0(1),\end{equation}
i.e., the integral is bounded by a constant $C(m)$ independent of $n$ and ${\bf u}$ for $|{\bf u}|=1$. For a similar argument shows the same is true for 
$$\int_{\{g < -(m+1)\log n\}}g({\bf a})dProb_{m_n}({\bf a}).$$ Then, since $Prob_{m_n}$ is a probability measure,
$$|\int_{\{-(m+1)\log n \leq g \leq (m+1)\log n\}}g({\bf a})dProb_{m_n}({\bf a})|=0(\log n)$$
independent of ${\bf u}$ for $|{\bf u}|=1$ and (\ref{geqn}) follows.

To prove (\ref{geqn2}), from Corollary \ref{forlater}, for $k=m+1,m+2,...,$ we have
\begin{equation} \label{dk1}Prob_{m_n}\{{\bf a}\in \CC^{m_n}: ||{\bf a}||\geq n^k\}\leq T\bigl(\frac{m_n}{n^k}\bigr)^2.\end{equation}
Note that $m_n=0(n^m)$ so $\frac{m_n}{n^k}=0(n^{m-k})$.
We set 
$$d_k=Prob_{m_n}\{{\bf a}: g({\bf a})>k\log n\}.$$
Since $|{\bf u}|=1$, from (\ref{dk1}) we have 
\begin{equation}\label{dk}d_k\leq T(\frac{m_n}{n^k})^2.\end{equation}

 We break up the region $\{g > (m+1)\log n\}$ into the sets 
$$D_k:=\{k\log n < g \leq (k+1)\log n\}, \ k=m+1,m+2,....$$
Then $Prob_{m_n}(D_k)=d_k-d_{k+1},$ and
 $$ \int_{\{g > (m+1)\log n\}}g({\bf a})dProb_{m_n}({\bf a})\leq \sum_{k=m+1}^{\infty} (k+1)\log n (d_k-d_{k+1})$$
$$=\log n\sum_{k=m+1}^{\infty}\bigl( (k+1)(d_k-d_{k+1})\bigr).$$
Now
$$\sum_{k=m+1}^{\infty}\bigl( (k+1)(d_k-d_{k+1})\bigr)=(m+2)d_{m+1} +\sum_{k=m+2}^{\infty} d_k.$$
Using (\ref{dk}), 
$$(m+2)d_{m+1} +\sum_{k=m+2}^{\infty}d_k\leq (m+2)T\bigl(\frac{m_n}{n^{m+1}}\bigr)^2 +\sum_{k=m+2}^{\infty} T\bigl(\frac{m_n}{n^k}\bigr)^2.$$
This gives 
$$\int_{\{g > (m+1)\log n\}}g({\bf a})dProb_{m_n}({\bf a})\leq C\frac{\log n}{n^2}$$
with a constant $C=C(m)$ independent of $n$ and $\bf{u}$ for $|{\bf u}|=1$, verifying (\ref{geqn2}).

\end{proof}

\begin{corollary} \label{expzero} With $a_{\nu}^{(n)}$ being i.i.d. real random variables with real Gaussian distribution $\phi(x)=\frac{1}{\sqrt \pi}e^{-x^2}$, we have 
\begin{equation}\label{expzeroE}\lim_{n\to \infty} E(\tilde Z_{H_n})=dd^cV_{K,Q}.\end{equation}
\end{corollary}

For $k=2,3,...,m$ we use the deterministic result together with the fact that 
$$E(\tilde Z^k_{{\bf H}_n})=E(\frac{1}{n^k}dd^c\log |H_n^{(1)}(z)|\wedge \cdots \wedge dd^c\log |H_n^{(k)}(z)|)$$
$$=E(\tilde Z_{H_n^{(1)}})\wedge \cdots \wedge E(\tilde Z_{H_n^{(k)}})$$
to deduce the following.

\begin{corollary} \label{expzero1} Under the hypotheses of Proposition \ref{propr} (and hence Theorem  \ref{keythm}), we have 
$$\lim_{n\to \infty} E(\tilde Z^k_{{\bf H}_n})=(dd^cV_{K,Q})^k$$
for $k=2,3,...,m$.

\end{corollary}

\section{\bf Open problems.} \label{sec:op}
In this section, we state some open problems.
\begin{enumerate}
\item In the context of random polynomial mappings, under the hypotheses of Theorem \ref{pointmap}, do we have 
$$\lim_{n\to \infty}\bigl[dd^c\bigl(\frac{1}{n}\log ||F_n(z)||\bigr)]^k=\bigl(dd^cV_{K,Q}^*(z)\bigr)^k$$
{\it almost surely in $\mathcal F$} as positive currents for $k=2,...,m$?\\
The case of Gaussian coefficients was done by Shiffman \cite{S}.
\item In the random polynomial mapping setting, if $k=m=2$, it follows from a result of Blocki \cite{Zbig} that if 
\begin{equation}\label{w12}\lim_{n\to \infty}\frac{1}{n}\log ||F_n(z)||=V_{K,Q}^*(z) \ \hbox{in} \ W^{1,2}_{loc}(\CC^2)\end{equation} then  
$$\lim_{n\to \infty}\bigl(dd^c\bigl(\frac{1}{n}\log ||F_n(z)||\bigr)\bigr)^2=\bigl(dd^cV_{K,Q}^*(z)\bigr)^2$$
as positive currents. Do we have (\ref{w12}) {\it almost surely in $\mathcal F$}?
 \item Consider on $\CC^m$ {\it non-random} polynomial mappings, i.e., 
$$F_n:=(H_n^{(1)}(z),...,H_n^{(m)}(z))$$ with
$$  H_n^{(j)}(z):=\sum_{|\nu|\leq n} a_{\nu}^{(n,j)}p_{\nu}^{(n)}(z)$$ and $a_{\nu}^{(n,j)}\in\CC$. Find conditions on the coefficients $a_{\nu}^{(n,j)}$ so that 
$$\lim_{n\to \infty}\bigl[dd^c\bigl(\frac{1}{n}\log ||F_n(z)||\bigr)]^m=\bigl(dd^cV_{K,Q}^*(z)\bigr)^m \ \hbox{weak*},$$
i.e., the normalized counting measure of the common zeros of the $m-$tuples of these polynomials converge to the Monge-Amp\`ere measure of the weighted pluricomplex Green function. For the case $m=1$, see \cite{Bloom}.

%\item In \cite{SZ}, Lemma 3.3, a variance estimate, with complex Gaussian randomization, is proved which in particular allows one to prove a. s. convergence of the normalized zero currents $\tilde Z^k_{{\bf H}_n}$ to $(dd^cV_{K,Q})^k$. Can one prove a similar variance estimate with {\it real} Gaussian randomization? 
%\item What conditions on $\phi$ -- perhaps strengthening (\ref{hyp1}) and (\ref{hyp2}) -- allow one to get a variance estimate sufficient to prove a. s. convergence of the normalized zero currents $\tilde Z^k_{{\bf H}_n}$ to $(dd^cV_{K,Q})^k$?

\end{enumerate}

\end{document}